\documentclass[12pt]{amsart}

\usepackage{amssymb, amscd, txfonts}
\usepackage{graphicx}

\numberwithin{equation}{section}

\newtheorem{theorem}{Theorem}[section]
\newtheorem{proposition}[theorem]{Proposition}
\newtheorem{lemma}[theorem]{Lemma}
\newtheorem{corollary}[theorem]{Corollary}

\theoremstyle{definition}

\theoremstyle{remark}
\newtheorem{remark}[theorem]{Remark}
\newtheorem{claim}[theorem]{Claim}

\newcommand{\N}{\mathbb{N}}
\newcommand{\Z}{\mathbb{Z}}

\newcommand{\Q}{\mathbb{Q}}

\newcommand{\C}{\mathbb{C}}
\newcommand{\HH}{\mathbb{H}}
\newcommand{\proj}{{\mathbb P}}

\newcommand{\SL}{{\rm SL}_2(\mathbb{Z})}
\newcommand{\Mp}{{\rm Mp}_2(\mathbb{Z})}
\newcommand{\Or}{{\rm O}^+}
\newcommand{\Ost}{\widetilde{{\rm O}}^+}

\newcommand{\DL}{\mathcal{D}_{L}}
\newcommand{\DK}{\mathcal{D}_{K}}
\newcommand{\FL}{\mathcal{F}_{L}}
\newcommand{\divi}{{\rm div}}
\newcommand{\ord}{{\rm ord}}
\newcommand{\Hast}{\mathcal{H}_{\ast}}

\begin{document}

\title[]{Finiteness of $2$-reflective lattices of signature $(2, n)$}
\author[]{Shouhei Ma}
\thanks{Supported by Grant-in-Aid for Scientific Research No.22224001.} 
\address{Department~of~Mathematics, Tokyo~Institute~of~Technology, Tokyo 152-8551, Japan}
\email{ma@math.titech.ac.jp}
\keywords{} 

\begin{abstract}
A modular form for an even lattice $L$ of signature $(2, n)$ is said to be $2$-reflective 
if its zero divisor is set-theoretically contained in the Heegner divisor defined by the $(-2)$-vectors in $L$. 
We prove that there are only finitely many even lattices with $n\geq7$ which admit $2$-reflective modular forms. 
In particular, there is no such lattice in $n\geq26$ except the even unimodular lattice of signature $(2, 26)$. 
This proves a conjecture of Gritsenko and Nikulin in the range $n\geq7$. 
\end{abstract} 

\maketitle


\section{Introduction}\label{sec:intro}

The concept of reflective modular forms for orthogonal groups of signature $(2, n)$ 
seems to have been first found in the works of Borcherds \cite{Bo1}, \cite{Bo2} and Gritsenko-Nikulin \cite{G-N1}, \cite{G-N2}.  
Through their work, reflective modular forms have been recognized as a key item in various related topics, such as  
the classification of Lorentzian Kac-Moody algebras \cite{G-N1}, \cite{G-N2}, \cite{G-N4}, \cite{Sch}; 
search of interesting hyperbolic reflection groups \cite{Bo4}; 
mirror symmetry for K3 surfaces \cite{G-N0}, \cite{G-N3}; 
and more recently search of modular varieties of non-general type \cite{Gr}, \cite{G-H}. 
Although there are some variations in the definition, 
those authors have shared the belief that such modular forms should be rare. 
This in turn reflects the expectation that the above objects would be exceptional. 
In this article we study \textit{$2$-reflective modular forms}, 
the most basic class of reflective modular forms, 
and prove that the number of lattices possessing such modular forms is finite in the range $n\geq7$. 

Let $L$ be an even lattice of signature $(2, n)$ with $n\geq3$. 
The Hermitian symmetric domain ${\DL}$ of type IV attached to $L$ is a connected component of the space 
\begin{equation*}\label{eqn:def type IV domain}
\{  {\C}\omega \in {\proj}(L\otimes{\C}) \; | \; (\omega, \omega)=0, (\omega, \bar{\omega})>0  \} . 
\end{equation*}
We write ${\Or}(L)$ for the subgroup of the orthogonal group ${\rm O}(L)$ preserving ${\DL}$. 
Let 
\begin{equation*}\label{eqn:-2Heegner}
\mathcal{H} = \bigcup_{\begin{subarray}{c} l\in L \\  (l, l)=-2 \end{subarray}} l^{\perp} \cap {\DL}
\end{equation*}
be the Heegner divisor of ${\DL}$ defined by the $(-2)$-vectors in $L$. 
A holomorphic modular form on ${\DL}$ with respect to a finite-index subgroup $\Gamma < {\Or}(L)$ 
and a character $\chi:\Gamma\to{\C}^{\times}$ is said to be $2$-reflective if its zero divisor is 
set-theoretically contained in $\mathcal{H}$. 
Note that this definition requires to specify the lattice $L$, not only the modular group $\Gamma$. 
In some cases, such a modular form has the geometric interpretation as 
characterizing discriminant locus in the moduli of lattice-polarized $K3$ surfaces. 
Following \cite{G-N2}, \cite{G-N3}, we say that the lattice $L$ is $2$-reflective if it admits 
a (non-constant) $2$-reflective modular form for some $\Gamma$ and $\chi$. 

\begin{theorem}\label{main}
There are only finitely many $2$-reflective lattices of signature $(2, n)$ with $n\geq7$. 
In particular, there is no $2$-reflective lattice in $n\geq26$ except the even unimodular lattice $II_{2,26}$ of signature $(2, 26)$. 
\end{theorem}

This proves in the range $n\geq7$ a conjecture of Gritsenko-Nikulin in \cite{G-N3}  (part (a) of ``Arithmetic Mirror Symmetry Conjecture''), 
which seems to be first formulated in \cite{Ni2}. 
For $n=4, 5, 6$ we still have the finiteness under some condition 
(Corollaries \ref{cor:finiteness under l<n-3} and \ref{cor:finiteness variation}). 
When $n=3$, Gritsenko and Nikulin gave some classification in \cite{G-N2} \S 5.2 and \cite{G-N4} \S 2. 
 
Gritsenko-Nikulin's conjecture in \cite{G-N3} is a special version of 
their more general finiteness conjecture for reflective modular forms \cite{G-N1}. 
Scheithauer \cite{Sch} classified reflective modular forms of singular weight with vanishing order $\leq1$, 
for lattices of square-free level.  
Looijenga \cite{Lo} proved another part of Gritsenko-Nikulin's conjectures in \cite{G-N1}, 
which might also give an approach to the classification of reflective modular forms. 

Theorem \ref{main} is essentially an effective result. 
This means that it would be in principal possible with the method here to 
write down finitely many lattices which include all $2$-reflective ones in $n\geq7$ (cf.~Remark \ref{remark:effective bound}). 
However, this should still contain a large redundancy. 
 
We will first prove in \S \ref{sec:n>25} the second sentence of Theorem \ref{main} using the theory of Borcherds products. 
The argument works more generally for reflective modular forms. 
The unimodular lattice $II_{2,26}$ carries Borcherds' $\Phi_{12}$ function (\cite{Bo1}) as a 2-reflective modular form of weight $12$. 
This is the ``last'' $2$-reflective form. 

After thus establishing a bound of $n$, we then prove in \S \ref{sec:fixed n} the first sentence of Theorem \ref{main} for each fixed $n$. 
By a result of \cite{Ma} the finiteness problem reduces to the boundedness of relative vanishing orders of $2$-reflective forms. 
We can obtain estimates of those orders by restricting the $2$-reflective forms to certain modular curves. 
The key point is that we can always choose such curves from a certain finite list.  

The proof of Theorem \ref{main} is thus a sandwich by two approaches for classifying $2$-reflective lattices: 
via the theory of Borcherds products, and via the volume of $(-2)$-Heegner divisors. 

\S \ref{sec:reduction} is of preliminary nature, 
where we set up some basic reduction technique and a few notation.


\vspace{0.4cm}
\noindent
\textbf{Acknowledgements.} 
The author would like to thank Prof. K.~Yoshikawa for introducing him to this problem. 
The referee has suggested that the bound of $n$ could be improved to the present value, 
and has taught us the proof that is produced here. 
We wish to express our deep gratitude to him for his invaluable suggestion.


\section{Basic reductions}\label{sec:reduction}

In order to prove Theorem \ref{main}, it will be useful to narrow the class of $2$-reflective modular forms 
that we actually deal with. 
Throughout $L$ will be an even lattice of signature $(2, n)$ with $n\geq3$. 
We first kill the characters $\chi$. 

\begin{lemma}
Suppose that $L$ admits a $2$-reflective form 
with respect to $\Gamma<{\Or}(L)$ and a character $\chi:\Gamma\to{\C}^{\times}$. 
Then $L$ also has a $2$-reflective form with respect to $\Gamma$ and the trivial character. 
\end{lemma}

\begin{proof}
Let $F$ be the given $2$-reflective form. 
Since the abelianization of $\Gamma$ is finite (\cite{Mar} Proposition 6.19 in p.~333), 
the image of $\chi$ is of finite order, say $d$. 
Hence $F^d$ is a modular form with respect to $\Gamma$ and the trivial character. 
It is obvious that $F^d$ is also $2$-reflective. 
\end{proof}

We will not mention the character when it is trivial. 
Next we are free to change the arithmetic group $\Gamma$ inside ${\Or}(L)$. 

\begin{lemma}\label{lemma:reduction group}
Suppose that $L$ admits a $2$-reflective form with respect to some $\Gamma<{\Or}(L)$. 
Then $L$ also has a $2$-reflective form with respect to any other finite-index subgroup $\Gamma'<{\Or}(L)$. 
\end{lemma} 

\begin{proof}
Let $F$ be the given $2$-reflective form, 
and write $\Gamma''=\Gamma\cap\Gamma'$.  
We choose representatives $\gamma_1, \cdots, \gamma_{\delta}\in\Gamma'$ of the coset $\Gamma''\backslash\Gamma'$ 
and take the product 
\begin{equation*}
F' = \prod_{i=1}^{\delta} (F|_{\gamma_i}). 
\end{equation*}
This is a modular form with respect to $\Gamma'$. 
Since each $\gamma_i$ preserves $L$, 
the zero divisor of $F|_{\gamma_i}$ is contained in $\gamma_i^{\ast}\mathcal{H}=\mathcal{H}$. 
Hence $F'$ is $2$-reflective. 
\end{proof}

It is convenient to normalize the choice of modular group to the following one. 
Let $L^{\vee}$ be the dual lattice of $L$ and $A_L=L^{\vee}/L$ be the discriminant group, 
which is canonically equipped with the ${\Q}/{\Z}$-valued quadratic form defined by 
$(\lambda, \lambda)/2$ mod ${\Z}$ for $\lambda\in L^{\vee}/L$. 
The subgroup of ${\Or}(L)$ that acts trivially on $A_L$ is denoted by ${\Ost}(L)$. 
If $F$ is a $2$-reflective form on ${\DL}$ with respect to ${\Ost}(L)$ and the trivial character, 
we shall say simply that $F$ is a $2$-reflective form for $L$ and that $L$ has the $2$-reflective form $F$. 
By the above lemmas, such $F$ exists exactly when $L$ is $2$-reflective in the sense of \S \ref{sec:intro}. 

So far we have fixed the reference lattice $L$. 
We can reduce the lattice in the following way. 

\begin{lemma}\label{lemma:reduction lattice}
Suppose that $L$ has a $2$-reflective form. 
Then any even overlattice $L'$ of $L$ has a $2$-reflective form too. 
\end{lemma}

\begin{proof}
Let $F$ be the given $2$-reflective form for $L$. 
Since $L\otimes{\Q}=L'\otimes{\Q}$, we can identify ${\DL}$ with $\mathcal{D}_{L'}$ canonically. 
Since $(-2)$-vectors in $L$ are also $(-2)$-vectors in $L'$, 
the $(-2)$-Heegner divisor on ${\DL}$ is contained in the $(-2)$-Heegner divisor on $\mathcal{D}_{L'}$ 
under this identification. 
Hence we may view $F$ as a $2$-reflective form on $\mathcal{D}_{L'}$ with reference lattice $L'$, 
with respect to ${\Ost}(L)<{\Or}(L'\otimes{\Q})$. 
Since $L'$ is contained in $L^{\vee}$, any element of ${\Ost}(L)$ preserves $L'$ and acts trivially on $A_{L'}$. 
Hence we have ${\Ost}(L)\subset{\Ost}(L')$, 
and the claim follows from Lemma \ref{lemma:reduction group} applied to $L'$. 
\end{proof}

\begin{corollary}\label{cor:reduction lattice}
If $L$ is not $2$-reflective, neither is any finite-index sublattice of $L$. 
\end{corollary}

We have the following decomposition \eqref{eqn:decomp Heegner div ver I} of the $(-2)$-Heegner divisor $\mathcal{H}$. 
For $\lambda\in A_L$ and $m\in{\Q}$ we write 
\begin{equation*}
\mathcal{H}(\lambda, m) = \bigcup_{\begin{subarray}{c} l\in L+\lambda \\  (l, l)=2m \end{subarray}} l^{\perp} \cap {\DL}
\end{equation*}
for the Heegner divisor of discriminant $(\lambda, m)$. 
In particular, $\mathcal{H}=\mathcal{H}(0, -1)$. 
Let $\pi_L\subset A_L$ be the subset of elements of order $2$ and norm $-1/2$. 
For each $\mu\in \pi_L$ we write $\mathcal{H}_{\mu} = \mathcal{H}({\mu}, -1/4)$. 
We also set 
\begin{equation*}\label{eqn:nonsplit Heegner div} 
\mathcal{H}_0 = \bigcup_{\begin{subarray}{c} l\in L, \; (l, l)=-2 \\ {\divi}(l)=1 \end{subarray}} l^{\perp} \cap {\DL},  
\end{equation*}
where for a primitive vector $l\in L$, ${\divi}(l)$ is the natural number generating the ideal $(l, L)$. 
Then we have the decomposition 
\begin{equation}\label{eqn:decomp Heegner div ver I}
\mathcal{H} = \mathcal{H}_0 + \sum_{\mu\in\pi_L}\mathcal{H}_{\mu} 
\end{equation}
with no common component between different ${\Hast}$. 
When the lattice $L$ contains $2U$,  
the Eichler criterion (cf.~\cite{Sc} \S 3.7) tells that 
each ${\Hast}$ is an ${\Ost}(L)$-orbit of a single quadratic divisor $l^{\perp}\cap{\DL}$. 

We also prefer the algebro-geometric setting: 
consider the quotient space ${\FL}={\Ost}(L)\backslash{\DL}$, 
which is a quasi-projective variety of dimension $n$. 
Let $H, H_0, H_{\mu}\subset{\FL}$ be the algebraic divisors given by 
$\mathcal{H}, \mathcal{H}_0, \mathcal{H}_{\mu}$ respectively. 
Then we have 
\begin{equation*}\label{eqn:decomp Heegner div ver II}
H = H_0 + \sum_{\mu\in\pi_L}H_{\mu}, 
\end{equation*}
and the Eichler criterion says that each $H_{\ast}$ is irreducible when $L$ contains $2U$. 
The stabilizer of a $(-2)$-vector $l\in L$ in ${\Ost}(L)$, 
viewed as a subgroup of ${\Or}(K)$ for the orthogonal complement $K=l^{\perp}\cap L$, 
contains ${\Ost}(K)$ by \cite{Ni1}. 
Therefore, when $l$ defines the component $H_{\ast}$, we have a finite morphism $\mathcal{F}_K\to H_{*}$.


\section{Absence in higher dimension}\label{sec:n>25}

In this section we prove the second sentence of Theorem \ref{main}. 
In an earlier version of this paper we proved it only for $n\geq30$ using Borcherds' duality theorem \cite{Bo4}; 
the proof presented below was suggested by the referee.  

\begin{proposition}\label{prop:non-exceptional n>26}
Let $L$ be an even lattice of signature $(2, n)$ with $n\geq26$ and $L\not\simeq II_{2,26}$.  
Then $L$ is not $2$-reflective. 
\end{proposition}

\begin{proof}
We first prove the assertion under the assumption that $L$ contains $2U$. 
In that case, if $L$ has a $2$-reflective form $F$, its divisor can be written as 
\begin{eqnarray*}
{\divi}(F) & = & \beta_0\mathcal{H}_0 + \sum_{\mu\in\pi_L} \beta_{\mu}\mathcal{H}_{\mu} \\ 
& =  & \beta_0\mathcal{H} + \sum_{\mu\in\pi_L} (\beta_{\mu}-\beta_0) \mathcal{H}_{\mu}
\end{eqnarray*}
for some nonnegative integers $\beta_*$. 
By the theorem of Bruinier (\cite{Br} \S 5.2, and also \cite{Br2}), $F$ is a Borcherds product: 
there exists a nearly holomorphic modular form $f(\tau)$ of weight $1-n/2$ and type $\rho_L$ for ${\Mp}$ with principal part 
\begin{equation*}
\beta_0\, q^{-1}\mathbf{e}_0 + \sum_{\mu\in\pi_L} (\beta_{\mu}-\beta_0)\, q^{-1/4}\mathbf{e}_{\mu}, 
\end{equation*}
such that $F$ is the Borcherds lift of $f$. 
Here $q=e^{2\pi i\tau}$, 
${\Mp}$ is the metaplectic cover of ${\SL}$, 
$\rho_L$ the Weil representation of ${\Mp}$ on the group ring ${\C}[A_L]$, 
and $\mathbf{e}_{\lambda}\in{\C}[A_L]$ the basis vector corresponding to the element $\lambda\in A_L$.  

Consider the product $f(\tau) \Delta(\tau)$ with the classical cusp form $\Delta(\tau)$ of weight $12$. 
This is a nearly holomorphic modular form of weight $13-n/2$ and type $\rho_L$ with Fourier expansion 
\begin{equation}\label{eqn:Fourier expansion of fDelta}
\beta_0 \mathbf{e}_0 + \textrm{(higher power of $q$)}. 
\end{equation}
In particular, $f \Delta$ is holomorphic at the cusp. 
When $n\geq27$, the weight $13-n/2$ is negative and thus $f \Delta\equiv 0$, so $f\equiv0$. 
When $n=26$, $f \Delta$ has weight $0$ and hence must be (constantly) an ${\Mp}$-invariant vector in ${\C}[A_L]$. 
By \eqref{eqn:Fourier expansion of fDelta} this vector should be $\beta_0\mathbf{e}_0$. 
On the other hand,  $\mathbf{e}_0$ is mapped to $|A_L|^{-1/2}\sum_{\lambda\in A_L}\mathbf{e}_{\lambda}$ 
by the action of the element $\left( \begin{pmatrix}0 & -1 \\ 1 & 0\end{pmatrix}, \sqrt{\tau}\right)$ of ${\Mp}$, 
hence cannot be ${\Mp}$-invariant when $|A_L|\ne 1$. 
Therefore $f\equiv 0$. 

Now consider the general case $L$ does not necessarily contain $2U$. 
Take a maximal even overlattice $L'\supset L$. 
Then $A_{L'}$ is anisotropic and hence has length $\leq3$. 
By Nikulin's theory \cite{Ni1}, we see that $L'$ contains $2U$. 
If $L'\not\simeq II_{2,26}$, $L'$ is not $2$-reflective by our first step, and neither is $L$ by Corollary \ref{cor:reduction lattice}. 
In case $L'\simeq II_{2,26}$, we choose an intermediate lattice $L\subset L''\subset II_{2,26}$ such that $II_{2,26}/L''$ is cyclic. 
Then $A_{L''}$ has length $\leq 2$ and so $L''$ contains $2U$, and we can apply our first step to $L''$. 
\end{proof}

Let us remark that the first part of this proof works more generally for reflective modular forms. 
A modular form on ${\DL}$ is called \textit{reflective} if its divisor is 
set-theoretically contained in the union of quadratic divisors $l^{\perp}\cap{\DL}$ defined by reflective vectors $l$ of $L$. 

\begin{proposition}\label{prop:non-exceptional reflective}
Let $L$ be an even lattice of signature $(2, n)$ with $n\geq26$ containing $2U$. 
Assume that $L$ is not isometric to $II_{2,26}$. 
Then there is no reflective modular form for ${\Ost}(L)$. 
\end{proposition}

\begin{proof}
A primitive vector $l\in L$ with $(l, l)=-2d$ is reflective if and only if ${\divi}(l)=2d$ or $d$. 
If we write $\lambda=[l/{\divi}(l)]\in A_L$, $l^{\perp}\cap{\DL}$ belongs to $\mathcal{H}(\lambda, -1/4d)$ in the first case, and to 
\begin{equation*}
\mathcal{H}(\lambda, -1/d) - \sum_{2\lambda'=\lambda}\mathcal{H}(\lambda', -1/4d)
\end{equation*}
in the second case. 
By the Eichler criterion and Bruinier's theorem, a reflective modular form is the Borcherds lift of a nearly holomorphic modular form 
with principal part 
\begin{equation*}
\sum_{d\in{\N}} \sum_{\begin{subarray}{c} \lambda\in A_L, {\ord}(\lambda)=d \\ (\lambda, \lambda)=-2/d \end{subarray}} 
\left( \beta_{\lambda}q^{-1/d}\mathbf{e}_{\lambda} + \sum_{2\lambda'=\lambda} \beta_{\lambda'}q^{-1/4d}\mathbf{e}_{\lambda'} \right). 
\end{equation*}
The singularity is of order at most $1$, and the coefficient of $q^{-1}$ is $\beta_0\mathbf{e}_0$. 
Thus we can argue similarly. 
\end{proof}


\section{Finiteness at each dimension}\label{sec:fixed n}

In this section we prove that with $n\geq7$ fixed, 
there are only finitely many $2$-reflective lattices of signature $(2, n)$. 
In \S \ref{ssec:slope} we first show that the finiteness follows if we could universally bound 
the orders of zero of $2$-reflective forms relative to the weights, 
using a result from \cite{Ma} and under the condition that the lattices contain $2U$. 
Then we give an estimate of such relative orders 
by restricting the modular forms to ``general'' modular curves on ${\FL}$ 
and appealing to the well-known situation on the curves. 
In \S \ref{ssec:pool} we obtain a desired bound by proving that 
we can choose such modular curves always from a finite list, under an assumption on $A_L$. 
This deduces the finiteness under that condition. 
If $n\geq7$, we can remove the assumption by a lattice-theoretic argument.

\subsection{Maximal slope}\label{ssec:slope}

Let us assume in this subsection that the lattices $L$ contain $2U$. 
Let $F$ be a $2$-reflective form for $L$ of weight $\alpha$. 
As in the proof of Proposition \ref{prop:non-exceptional n>26}, we can write 
\begin{equation*}
{\divi}(F) = \beta_0\mathcal{H}_0 + \sum_{\mu\in\pi_L} \beta_{\mu}\mathcal{H}_{\mu}
\end{equation*}
for some nonnegative integers $\beta_*$. 
We shall call ${\max}_{*}(\beta_{*}/\alpha)$ the \textit{maximal slope} of $F$ where $*\in \{ 0, \pi_L \}$, 
and denote it by $\lambda(F)$. 
The reason to consider this invariant is the following.  

\begin{proposition}\label{prop:finiteness under bounded slope}
Fix a positive rational number $\lambda$. 
Then there are only finitely many even lattices $L$ of signature $(2, n)$ with $n\geq3$ and containing $2U$ 
such that $L$ has a $2$-reflective form $F$ with $\lambda(F)\leq \lambda$. 
\end{proposition}
 
\begin{proof}
Suppose we have a $2$-reflective form $F$ as above. 
If $\mathcal{L}$ denotes the ${\Q}$-line bundle over ${\FL}$ of modular forms of weight $1$, 
then the ${\Q}$-divisor $(\lambda/2)H-\mathcal{L}$ of ${\FL}$ is ${\Q}$-effective. 
In particular, $\lambda^{-1}\mathcal{L}-H/2$ is not big. 
According to \cite{Ma} Theorem 1.4, there are only finitely many even lattices containing $2U$ with this property. 
\end{proof}

\begin{remark}\label{remark:effective bound}
Given $\lambda$, one would be able to enumerate all \textit{possible} $L$ as in the proposition: 
they should satisfy the inequality 
\begin{equation*}
\sqrt{|A_L|}  \; < \;  (n\lambda/2) \cdot \left( 1+\lambda \right)^{n-1} \cdot \{ \, 9\, f_{AI}(n)+2^{n-2} f_{AII}(n) \, \}, 
\end{equation*}
where $f_{*}(n)$ are the functions defined in \cite{Ma} \S 5.3. 
This estimate is a weakest version and could be improved when the class of lattices is specified. 
See \cite{Ma} \S 3, \S 4.2, \S 4.3 and \S 5.3 for more detail. 
For instance, when $F$ does not vanish at $\mathcal{H}_0$, the term $2^{n-2} f_{AII}(n)$ can be removed. 
\end{remark}

A natural approach to estimate the maximal slopes is to consider restriction to modular curves. 
Let $F$ be a $2$-reflective form for $L$ and $\mathcal{H}_{*}\subset\mathcal{H}$ be the component 
where $F$ attains its maximal slope where $*\in \{ 0, \pi_L \}$. 
\textit{Assume} that we have a sublattice $K\subset L$ of signature $(2, 1)$, not necessarily primitive, 
that satisfies the following ``genericity'' conditions: 
\begin{enumerate}
\item[(i)] for every $(-2)$-vector $l\in L$ we have $(l, K)\not\equiv 0$; 
\item[(ii)] there exists a $(-2)$-vector $l'\in L$ which defines a component of $\mathcal{H}_{*}$ and 
whose orthogonal projection to $K\otimes{\Q}$ has negative norm. 
\end{enumerate}
The $1$-dimensional submanifold 
${\DK}=K^{\perp}\cap{\DL}$ of ${\DL}$ is naturally isomorphic to the upper half plane ${\HH}\subset{\proj}^1$ 
through its quadratic embedding. 
The condition (i) implies that the restriction $f=F|_{\DK}$ is not identically zero. 
$f$ is a modular form on ${\DK}$ with respect to ${\Ost}(K)$. 
Note that as a modular form on ${\HH}$, the weight of $f$ is twice the weight of $F$. 
Next the condition (ii) says that ${\DK}$ has intersection with $\mathcal{H}_*$, 
so the pullback $D_*=\mathcal{H}_*|_{\DK}$ is a nonzero divisor on $\mathcal{D}_K\simeq\mathbb{H}$. 

\begin{lemma}\label{lem:estimate slope via curve}
Assume that we have a sublattice $K\subset L$ of signature $(2, 1)$ satisfying the conditions (i) and (ii) above. 
Then there exists a constant $\lambda_K<\infty$ depending only on the isometry class of $K$ such that 
\begin{equation*}
\lambda(F) \leq \lambda_K. 
\end{equation*}
\end{lemma}

\begin{proof}
Let $f=F|_{\DK}$ and $D_*=\mathcal{H}_*|_{\DK}$ as above. 
If ${\rm wt}(f)$ is the weight of $f$ as a modular form on ${\HH}$, 
we have 
\begin{equation*}
{\divi}(f) = {\rm wt}(f)/2 \cdot \lambda(F) \cdot D_* + D 
\end{equation*}
for some effective, ${\Ost}(K)$-invariant divisor $D$ on ${\DK}$. 
We view this equality on a fundamental domain of ${\Ost}(K)$ for ${\DK}$. 
If we write ${\rm deg}'({\divi}(f))$ simply for the sum of orders of zeros of $f$ over the points of the fundamental domain, 
we thus obtain 
\begin{equation*}
\lambda(F) \; \leq \; \frac{2\, {\rm deg}'({\divi}(f))}{{\rm wt}(f)}. 
\end{equation*}
It is well-known that the right hand side can bounded only in terms of geometric invariants of the compactification of 
$\mathcal{F}_K$ such as genus, number of cusps and orders of stabilizers 
(cf.~\cite{Sh} Proposition 2.16). 
\end{proof}

\subsection{A finite pool of modular curves}\label{ssec:pool}

For a prime $p$ we write $l(A_L)_p$ for the length of the $p$-component of the discriminant group $A_L$. 
Let us assume until Corollary \ref{cor:finiteness under l<n-3} that $L$ satisfies 
\begin{equation}\label{eqn:length A_L condition}
l(A_L)_2\leq n-3, \qquad l(A_L)_p\leq n-4 \; \:  \textrm{for} \; \:  p>2. 
\end{equation}
By the result of Nikulin \cite{Ni1} such lattices $L$ contain $2U$, so that the results of \S \ref{ssec:slope} can be applied. 
We shall construct, with $n$ fixed, 
a finite set $\mathcal{P}_n$ of isometry classes of even lattices of signature $(2, 1)$ such that 
for \textit{arbitrary} $F$ and $L$ we can find a sublattice $K\subset L$ 
satisfying the conditions (i), (ii) in \S \ref{ssec:slope} from this pool. 

Here is a notation: 
for an even lattice $N$ of any signature, we write $\Delta(N)$ for the set of $(-2)$-vectors in $N$, 
and $R(N)\subset N$ the sublattice generated by $\Delta(N)$. 

Let us first define $\mathcal{P}_n$. 
Let $\mathcal{R}_n$ be the set of isometry classes of negative-definite $(-2)$-root lattices of rank $\leq n-2$. 
Since the members of $\mathcal{R}_n$ are direct sums of $A, D, E$ lattices, 
$\mathcal{R}_n$ is a finite set. 
We define a positive integer $a_n$ through  
the maximal norms of integral vectors in $R\in\mathcal{R}_n$ that are not orthogonal to any $(-2)$-vector:   
\begin{equation*}
-a_n = 
\min_{R\in\mathcal{R}_n} [ \: \max_{m\in R} \{ \, (m, m) \: | \: (m, l)\ne0 \: \textrm{for all} \: l\in\Delta(R) \, \} \: ]. 
\end{equation*}
We also define a positive integer $b_n$ through 
the minimal norms of integral vectors in the Weyl chambers of $U\oplus kA_1$ with $0\leq k\leq n-2$: 
\begin{equation*}
b_n = 
\max_{0\leq k\leq n-2} [ \: \min_{\begin{subarray}{c}m\in U\oplus kA_1 \\ (m, m)>0 \end{subarray}} 
\{ \, (m, m) \: | \: (m, l)\ne0 \: \textrm{for all} \: l\in\Delta(U\oplus kA_1) \, \} \: ]. 
\end{equation*}
Then we set 
\begin{equation*}
\mathcal{P}_n' = 
\{ \langle 4 \rangle \oplus  \langle 4 \rangle \oplus  \langle -a \rangle \: | \: 0\leq a \leq a_n \},  
\end{equation*}
\begin{equation*}
\mathcal{P}_n'' = 
\{ U \oplus  \langle b \rangle \: | \: 0\leq b \leq b_n \},  
\end{equation*}
\begin{equation*}
\mathcal{P}_n = \mathcal{P}_n' \cup \mathcal{P}_n''. 
\end{equation*}

\begin{lemma}\label{prop:finite pool}
Let $n\geq4$ be fixed. 
Then for any even lattice $L$ of signature $(2, n)$ with the condition \eqref{eqn:length A_L condition} 
and any $2$-reflective form $F$ for $L$, 
we can find a sublattice $K\subset L$ that satisfies the genericity conditions (i) and (ii) in \S \ref{ssec:slope} 
and that is isometric to a member of $\mathcal{P}_n$. 
\end{lemma}

\begin{proof}
We first consider the case the maximal slope of $F$ is attained at $\mathcal{H}_0$. 
We write $L=U_1\oplus U_2\oplus M$ with $U_1, U_2$ two copies of $U$ and $M$ negative-definite of rank $n-2$. 
We denote by $e_1, f_1$ and $e_2, f_2$ the standard hyperbolic basis of $U_1$ and $U_2$ respectively. 
We have the following two possibilities: 
\begin{enumerate}
\item[(a)] $M$ contains a $(-2)$-vector $l$ with $(l, L)={\Z}$; 
\item[(b)] any $(-2)$-vector $l$ in $M$, if exists, satisfies $(l, L)=2{\Z}$. 
\end{enumerate}

In case (a), we choose a vector $m$ from the root lattice $R(M)$ 
that is not orthogonal to any $(-2)$-vector in $R(M)$ 
and that has the maximal norm among such vectors. 
By the definition of $a_n$ we have $(m, m)\geq -a_n$. 
Then we set 
\begin{equation*}
K = {\Z}(e_1+2f_1) \oplus  {\Z}(e_2+2f_2) \oplus {\Z}m. 
\end{equation*}
This lattice belongs to $\mathcal{P}_n'$. 
The orthogonal complement in $L$ is described as 
\begin{eqnarray*}
K^{\perp} 
& = & {\Z}(e_1-2f_1) \oplus  {\Z}(e_2-2f_2) \oplus (m^{\perp}\cap M) \\ 
& \simeq &  \langle -4 \rangle \oplus  \langle -4 \rangle \oplus  (m^{\perp}\cap M). 
\end{eqnarray*}
Since all $(-2)$-vectors of $M$ are contained in $R(M)$, 
the lattice $m^{\perp}\cap M$ contains no $(-2)$-vector by the definition of $m$. 
Hence $K^{\perp}$ contains no $(-2)$-vector, which verifies the condition (i). 
By assumption we have a $(-2)$-vector $l\in M$ with $(l, L)={\Z}$. 
This $l$ defines a component of $\mathcal{H}_0$. 
Since $(l, m)\ne0$ by the definition of $m$, 
the orthogonal projection of $l$ to ${\Q}m$ is not zero vector. 
This gives the condition (ii). 

In case (b), any $(-2)$-vector in $M$ generates an orthogonal direct summand. 
Hence we can write $M=kA_1\oplus M'$ with $0\leq k\leq n-2$ and $M'$ containing no $(-2)$-vector. 
We take a positive norm vector $m$ from $U_2\oplus kA_1$ 
that is not orthogonal to any $(-2)$-vector in $U_2\oplus kA_1$ 
and that has the minimal norm among such vectors. 
We have $(m, m)\leq b_n$ by the definition of $b_n$. 
Then we put  
\begin{equation*}
K = U_1 \oplus {\Z}m, 
\end{equation*}
which belongs to $\mathcal{P}_n''$. 
Its orthogonal complement in $L$ is described as 
\begin{equation*}
K^{\perp} = (m^{\perp}\cap(U_2\oplus kA_1)) \oplus M'. 
\end{equation*}
By construction, both $m^{\perp}\cap(U_2\oplus kA_1)$ and $M'$ contain no $(-2)$-vector. 
Hence $K^{\perp}$ has no $(-2)$-vector too, which implies (i). 
Since $K$ contains the $(-2)$-vector $e_1-f_1\in U_1$ which satisfies $(e_1-f_1, L)={\Z}$, we also have (ii).

Next we consider the case the $2$-reflective form $F$ attains its maximal slope 
at $\mathcal{H}_{\mu}$ with $\mu\in\pi_L$. 

\begin{claim}
If $l\in L$ is a $(-2)$-vector with $(l, L)=2{\Z}$, 
then $l^{\perp}\cap L$ contains $2U$. 
\end{claim}

\begin{proof}
Since we have the splitting $L={\Z}l\oplus L'$ where $L'=l^{\perp}\cap L$, 
we see that  
$l(A_{L'})_2 = l(A_L)_2-1 \leq n-4$  
and $l(A_{L'})_p = l(A_L)_p \leq n-4$ for $p>2$. 
Then we can apply the result of Nikulin \cite{Ni1}. 
\end{proof}

\noindent
By this claim we can find a splitting $L=U_1\oplus U_2\oplus M$ such that 
$M$ contains a $(-2)$-vector which defines a component of $\mathcal{H}_{\mu}$. 
Then we can repeat the same construction as in the case (a) above. 
\end{proof}

By Lemmas \ref{lem:estimate slope via curve} and \ref{prop:finite pool}, the inequality 
\begin{equation*}
\lambda(F) \leq \max_{K\in\mathcal{P}_n} \, (\lambda_K) 
\end{equation*}
holds for any $2$-reflective form $F$ for lattices $L$ with the condition \eqref{eqn:length A_L condition}. 
Applying Proposition \ref{prop:finiteness under bounded slope}, we obtain the following. 

\begin{corollary}\label{cor:finiteness under l<n-3}
Let $n\geq4$ be fixed. 
Then there are only finitely many $2$-reflective lattices $L$ of signature $(2, n)$ with 
the condition \eqref{eqn:length A_L condition}. 
\end{corollary}

We note that in case the maximal slope is attained at $\mathcal{H}_0$, 
we used only the condition that $L$ contains $2U$ in the proof of Proposition \ref{prop:finite pool}. 
Hence we also have the following variant. 

\begin{corollary}\label{cor:finiteness variation}
Let $n\geq3$ be fixed. 
Then there are only finitely many even lattices of signature $(2, n)$ and containing $2U$ 
which has a $2$-reflective form whose maximal slope is attained at $\mathcal{H}_0$. 
\end{corollary}

This particularly applies to those $L$ containing $2U$ with $\pi_L=\emptyset$. 
When $n=3$, a more general classification is given in \cite{G-N2} \S 5.2 and \cite{G-N4} \S 2.

To deduce the finiteness for general lattices, we take overlattices following the next lemma. 

\begin{lemma}\label{lem:economic overlattice}
Let $A$ be a finite abelian group of exponent $e(A)$ endowed with a nondegenerate ${\Q}/2{\Z}$-valued quadratic form. 
Then there exists an isotropic subgroup $G\subset A$ such that 
$l(G^{\perp}/G)_2\leq 4$ and $l(G^{\perp}/G)_p\leq 3$ for $p>2$
and that the exponent of $G^{\perp}/G$ is equal to either $e(A)$ or $e(A)/2$. 
\end{lemma}

\begin{proof}
This is essentially obtained in the proof of \cite{Ma} Lemmas A.6 and A.7. 
We will not need to repeat that argument. 
\end{proof}

Now the proof of Theorem \ref{main} can be completed as follows. 
Let $n\geq7$ be fixed. 
By Corollary \ref{cor:finiteness under l<n-3} we have a natural number $e$ such that 
an even lattice $L$ of signature $(2, n)$ satisfying the condition \eqref{eqn:length A_L condition} is not $2$-reflective 
whenever $e(A_L)\geq e$. 
By Corollary \ref{cor:reduction lattice} and Lemma \ref{lem:economic overlattice}, 
any even lattice $L$ of signature $(2, n)$ with $e(A_L)\geq 2e$ is not $2$-reflective. 
We have only finitely many even lattices $L$ of signature $(2, n)$ with $e(A_L)<2e$, 
because those lattices have bounded discriminant.  
This concludes the proof of Theorem \ref{main}.



\begin{thebibliography}{99}


\bibitem{Bo1}Borcherds, R. 
\textit{Automorphic forms on $O_{s+2,2}(R)$ and infinite products.}
Invent. Math. \textbf{120} (1995), no. 1, 161--213. 

\bibitem{Bo2}Borcherds, R. 
\textit{Automorphic forms with singularities on Grassmannians.} 
Invent. Math. \textbf{132} (1998), no. 3, 491--562. 


\bibitem{Bo4}Borcherds, R. 
\textit{Reflection groups of Lorentzian lattices.} 
Duke Math. J. \textbf{104} (2000), no. 2, 319--366. 

\bibitem{Br}Bruinier, J.~H.
\textit{Borcherds products on $O(2, l)$ and Chern classes of Heegner divisors.} 
Lecture Notes in Math. \textbf{1780}. Springer-Verlag, 2002. 

\bibitem{Br2}Bruinier, J.~H.
\textit{On the converse theorem for Borcherds products.} 
J. Algebra \textbf{397} (2014), 315--342. 
 

\bibitem{Gr}Gritsenko, V. 
\textit{Reflective modular forms in algebraic geometry.}
arXiv:1005.3753. 

\bibitem{G-H}Gritsenko, V.; Hulek, K.
\textit{Uniruledness of orthogonal modular varieties.}
J. Algebraic Geom. \textbf{23} (2014), 711--725. 

\bibitem{G-N0}Gritsenko, V.; Nikulin, V. 
\textit{$K3$ surfaces, Lorentzian Kac-Moody algebras and mirror symmetry.} 
Math. Res. Lett. \textbf{3} (1996) 211--229.  
 
\bibitem{G-N1}Gritsenko, V.; Nikulin, V. 
\textit{Automorphic forms and Lorentzian Kac-Moody algebras. I.} 
Internat. J. Math. \textbf{9} (1998) no.2, 153--200.  
 
\bibitem{G-N2}Gritsenko, V.; Nikulin, V. 
\textit{Automorphic forms and Lorentzian Kac-Moody algebras. II.}
Internat. J. Math. \textbf{9} (1998), no. 2, 201--275.  
 
\bibitem{G-N3}Gritsenko, V.; Nikulin, V. 
\textit{The arithmetic mirror symmetry and Calabi-Yau manifolds.} 
Comm. Math. Phys. \textbf{210} (2000), no. 1, 1--11. 
 
\bibitem{G-N4}Gritsenko, V.; Nikulin, V. 
\textit{On the classification of Lorentzian Kac-Moody algebras.} 
Russian Math. Surveys \textbf{57} (2002), no. 5, 921--979. 

\bibitem{Lo}Looijenga, E. 
\textit{Compactifications defined by arrangements. II. Locally symmetric varieties of type IV.} 
Duke Math. J. \textbf{119} (2003), no. 3, 527--588. 

\bibitem{Ma}Ma, S. 
\textit{Finiteness of stable orthogonal modular varieties of non-general type.} 
arXiv:1309.7121. 

\bibitem{Mar}Margulis, G.~A. 
\textit{Discrete subgroups of semisimple Lie groups.} 
Ergeb. Math. Grenzgeb. (3) \textbf{17}. Springer-Verlag, Berlin, 1991. 

\bibitem{Ni1}Nikulin, V.V. 
\textit{Integral symmetric bilinear forms and some of their applications.}
Math. USSR Izv. \textbf{14} (1980), 103--167. 

\bibitem{Ni2}Nikulin, V.V. 
\textit{A remark on discriminants for moduli of K3 surfaces as sets of zeros of automorphic forms.} 
J. Math. Sci. \textbf{81} (1996), no. 3, 2738--2743. 

\bibitem{Sc}Scattone, F. 
\textit{On the compactification of moduli spaces for algebraic K3 surfaces.} 
Mem. Amer. Math. Soc. \textbf{70} (1987), no.~374. 

\bibitem{Sch}Scheithauer, N.~R. 
\textit{On the classification of automorphic products and generalized Kac-Moody algebras.} 
Invent. Math. \textbf{164} (2006), 641--678. 

\bibitem{Sh}Shimura, G. 
\textit{Introduction to the Arithmetic Theory of Automorphic Functions.} 
Iwanami Shoten/Princeton University Press, 1971. 

\end{thebibliography}
\end{document}